\documentclass[leqno,12pt]{article}
\usepackage[utf8]{inputenc}
\usepackage{graphicx}

\usepackage{xcolor}
\usepackage{hyperref}

\usepackage{amsmath}
\usepackage{amsfonts}
\usepackage{mathrsfs}

\usepackage{float}

\usepackage{tikz}
\usepackage[margin=1in]{geometry}

\usepackage{amsthm}

\newtheorem{thm}{Theorem}

\newtheorem{lem}{Lemma}[section]

\theoremstyle{definition}
\newtheorem{dfn}[lem]{Definition}

\numberwithin{equation}{section}

\newcommand{\xn}{(x_1,\dots,x_n)}

\newcommand{\xyz}{(x,y,z)}

\newcommand{\ov}[1]{\overline{#1}}

\newcommand{\ovx}{\ov{x}}
\newcommand{\ovy}{\ov{y}}
\newcommand{\ovz}{\ov{z}}

\newcommand{\ovxn}{(\ov{x_1},\dots,\ov{x_n})}
\newcommand{\ovxyz}{(\ov{x},\ov{y},\ov{z})}

\newcommand{\Cn}{C_n}
\newcommand{\Cthree}{C_3}

\newcommand{\NTn}{\operatorname{C}_n^*}
\newcommand{\NTthree}{\operatorname{C}_3^*}

\newcommand{\CthreeI}{C_3^{(I)}}
\newcommand{\CthreeII}{C_3^{(II)}}
\newcommand{\DnI}{D_n^{(I)}}
\newcommand{\DnII}{D_n^{(II)}}
\newcommand{\Dnstar}{D_n^{*}}

\newcommand{\Cthreeo}{C_3^{\text{o}}}

\newcommand{\vol}{\operatorname{vol}}

\newcommand{\sgn}{\operatorname{sign}}

\title{On cyclic and nontransitive probabilities}
\author{Pavle Vuksanovic\thanks{University of Illinois, 
email \url{pavlev2@illinois.edu}}
\and A.J. Hildebrand\thanks{University of Illinois, 
email \url{ajh@illinois.edu} (corresponding author)}}

\date{August 30, 2020}

\begin{document}

\maketitle
\begin{abstract}
Motivated by classical nontransitivity paradoxes, we call
an $n$-tuple $\xn \in[0,1]^n$ \emph{cyclic} if there exist independent
random variables $U_1,\dots, U_n$ with $P(U_i=U_j)=0$ for $i\not=j$ such
that $P(U_{i+1}>U_i)=x_i$ for $i=1,\dots,n-1$ and $P(U_1>U_n)=x_n$.
We call the tuple $\xn$ \emph{nontransitive} if it is cyclic and in
addition satisfies $x_i>1/2$ for all $i$.

Let $p_n$ (resp.~$p_n^*$) denote the probability that a randomly chosen
$n$-tuple $\xn\in[0,1]^n$ is cyclic (resp.~nontransitive).  We determine
$p_3$ and $p_3^*$ exactly, while for $n\ge4$ we give upper and lower
bounds for $p_n$ that show that $p_n$ converges to $1$ as $n\to\infty$.
We also determine the distribution of the smallest, middle, and largest
elements in a cyclic triple.
\end{abstract}

\section{Introduction}
\label{sec:introduction}

A classic example of a nontransitive probability paradox
is provided by the \emph{Efron dice}, a set of four dice invented
by Bradley Efron and popularized by Martin Gardner \cite{gardner1970}.
The Efron dice are six-sided dice with face values given as follows:
\begin{equation}
\label{eq:efron-dice}
\begin{split}
&A=\{0,0,4,4,4,4\},\quad
B=\{1,1,1,5,5,5\},\quad
\\
&C=\{2,2,2,2,6,6\},\quad
D=\{3,3,3,3,3,3\}.
\end{split}
\end{equation}
One can easily check that, with probability $2/3$ each, $B$ beats $A$,
$C$ beats $B$, and $D$ beats $C$ while, at the same time, $A$ beats $D$ with
probability $2/3$.  In this sense, the dice $A$, $B$, $C$, $D$ form a
nontransitive cycle.

More formally, if we let $A$, $B$, $C$, $D$ denote
independent discrete random variables that are uniformly distributed 
over the values listed in \eqref{eq:efron-dice}, then these variables
satisfy
\begin{equation}
\label{eq:efron-dice-probabilities}
P(B>A)=P(C>B)=P(D>C)=P(A>D)=\frac{2}{3}.
\end{equation}

Another classic example of nontransitivity is given by the 
following set of three-sided dice which seems to have first appeared 
(in a different, but equivalent, context) in a paper by Moon and Moser
\cite{moon-moser1967}:
\begin{equation}
\label{eq:moser-dice}
A=\{1,5,9\},\quad B=\{2,6,7\}, \quad C=\{3,4,8\}.
\end{equation}
The three-sided dice $A$, $B$, $C$ defined by 
\eqref{eq:moser-dice} form a nontransitive cycle with probabilities
\begin{equation}
\label{eq:moser-dice-probabilities}
P(B>A)=P(C>B)=P(A>C)=\frac{5}{9}.
\end{equation}

There now exists a large body of research motivated by such
nontransitivity phenomena; see \cite{angel-davis2018,
buhler2014, conrey2016, gilson2012, hulko2019, komisarski2021, lebedev2019,marengo2020} for some
recent work on this subject. 

It is natural to ask what the most ``extreme'' level of nontransitivity
is that can be achieved with constructions such as the Efron and
Moon-Moser dice.  Can one replace the probabilities $2/3$ and $5/9$ in
\eqref{eq:efron-dice-probabilities} and \eqref{eq:moser-dice-probabilities}
by even larger numbers?

To formalize this question, one can consider, for each
integer $n\ge3$, the quantity
\begin{equation}
\label{eq:pi-n-def}
\pi_n=\max \min\left(P(U_2>U_1),\dots,P(U_n>U_{n-1}), P(U_1>U_n)\right),
\end{equation}
where the maximum is taken over all sets of independent random variables 
$U_1,\dots,U_n$. 
The Efron and Moon-Moser dice constructions show that
$\pi_4\ge 2/3$ and $\pi_3\ge 5/9$.

The quantities $\pi_n$ were first investigated in the 1960s by  
Steinhaus and Trybula \cite{steinhaus-trybula1959}, Trybula
\cite{trybula1961,trybula1965}, Chang \cite{li-chien1961}, 
and Usiskin \cite{usiskin1964}
who showed that $\lim_{n\to\infty}\pi_n=3/4$ and determined the first few
values of $\pi_n$. It particular, the values of $\pi_3$ and $\pi_4$ are
given by 
\begin{equation}
\label{eq:pi-n-initial-values}
\pi_3=\frac{\sqrt{5}-1}{2},\quad \pi_4 = \frac{2}{3}.
\end{equation}
Thus the Efron dice construction is best-possible in the sense of
achieving the value $\pi_n$ for $n=4$; meanwhile, the Moon-Moser dice
construction for $n=3$ is not best-possible as $5/9<(\sqrt{5}-1)/2$.
More recently (see, e.g., \cite{bogdanov2010, komisarski2021, nagy2011})
it was shown that, for any $n\ge 3$,
\begin{equation}
\label{eq:pi-n-general-formula}
	\pi_n= 1-\frac{1}{4\cos^2(\pi/(n+2))}.
\end{equation}
Interestingly, the numbers $\pi_n$ defined by \eqref{eq:pi-n-def}
have come up independently in very different contexts such as graph theory
\cite{bondy2006, lovasz-pelikan1973, nagy2011} and theoretical computer science
\cite{kral2004, trevisan2004}. 

In this paper, we focus on another aspect of the nontransitivity
phenomenon that has received less attention in the literature, 
namely the question of which $n$-tuples  can be realized  as
a tuple of \emph{cyclic probabilities}
$(P(U_2>U_1),\dots,P(U_n>U_{n-1}), P(U_1>U_n))$, and how common such tuples are among all $n$-tuples in $[0,1]^n$.
We introduce the following definitions:

\begin{dfn}[Cyclic and nontransitive tuples]
\label{def:cyclic-tuples}
Let $n$ be an integer with $n\ge3$.
        \begin{itemize}
        \item[(i)]
An $n$-tuple $\xn \in[0,1]^n$ is called
                \emph{cyclic} if there exist independent random variables
                        $U_1,\dots, U_n$ with $P(U_i=U_j)=0$ for $i\not=j$
			such that
\begin{equation}
	\label{eq:cyclic-tuples}
	P(U_{i+1}>U_i)=x_i\quad  (i=1,\dots,n-1)
	\text{ and } P(U_1>U_n)=x_n.
\end{equation}
\item[(ii)]
        An $n$-tuple $\xn \in[0,1]^n$ is called
        \emph{nontransitive} if it is cyclic and satisfies
$x_i>1/2$ for all $i$.
\end{itemize}
\end{dfn}

The dice examples \eqref{eq:efron-dice} and \eqref{eq:moser-dice} 
show that the $4$-tuple $(2/3,2/3,2/3,2/3)$ and the triple $(5/9,5/9,5/9)$ 
are nontransitive and, hence, also cyclic. The $n$-tuple 
$(1/2,1/2,\dots,1/2)$ is cyclic for any $n\ge3$ as can be seen by
taking $U_1,\dots,U_n$ to be independent random variables with a continuous
common distribution.

On the other hand, the tuple $(1,1,\dots,1)$ is not cyclic.
Indeed, otherwise there would exist random variables $U_i$ such that,
with probability $1$, both    $U_1<U_2<\dots <U_n$ and $U_n<U_1$ hold.
But the first of these relations implies that $U_1<U_n$ holds with
probability $1$, contradicting the second relation.  A similar
contradiction arises whenever the components of $\xn$ are sufficiently
close to $1$.

The first non-trivial case of 
Definition \ref{def:cyclic-tuples}
is the case $n=3$, i.e., the case of cyclic
\emph{triples}.  For this case, Trybula \cite{trybula1961} and, independently, 
Suck \cite{suck2002}, 
gave necessary and sufficient conditions for a triple $\xyz$ to be cyclic 
(see Lemma \ref{lem:trybula-characterization} below).  For $n\ge 4$ 
a complete characterization of cyclic $n$-tuples is not known, though 
some partial results are known (see, for example, Trybula
\cite{trybula1965}). 

In this paper we consider the question of how common cyclic and
nontransitive tuples are among all tuples in the $n$-dimensional unit
cube $[0,1]^n$.  Let 
$p_n$ (resp.~$p_n^*$) be the probability 
that an $n$-tuple $\xn$ chosen randomly and
uniformly from $[0,1]^n$ is cyclic (resp.~nontransitive).  These probabilities
are given by the volumes of the regions $\Cn$ (resp.~$\NTn$)
of cyclic (resp.~nontransitive) $n$-tuples
inside the unit cube $[0,1]^n$; that is, we have
\begin{equation}
\label{eq:pn-definition}
p_n=\vol(\Cn), \quad p_n^*=\vol(\NTn),
\end{equation}
where 
\begin{align}
\label{eq:Cn-definition}
\Cn&=\{\xn\in[0,1]^n: \text{ $\xn$ is cyclic}\},
\\
\label{eq:NTn-definition}
\NTn&=\{\xn\in[0,1]^n: \text{ $\xn$ is nontransitive}\}.
\end{align}

In our first theorem, we determine the probabilities $p_3$ and
$p_3^*$ exactly.

\pagebreak[3]

\begin{thm}
\label{thm:main1}
\mbox{}
\begin{itemize}
\item[(i)] 
A random triple $(x_1,x_2,x_3)$ in $[0,1]^3$ is cyclic with probability 
\begin{align}
\label{eq:vol-cyclic3tuples}
        p_3&=\frac{11\sqrt{5}}{4}-\frac{17}{4}
	-6\ln(\sqrt{5}-1)\approx
 0.627575\dots.
\end{align}
\item[(ii)] 
A random triple $(x_1,x_2,x_3)$ in $[0,1]^3$ is nontransitive with probability 
\begin{align}
\label{eq:vol-nontransitive3tuples}
  p_3^*&=\frac{11\sqrt{5}}{8}-\frac{43}{16}
	-3\ln(\sqrt{5}-1)+\frac{3\ln 2}{8}
        \approx 0.011217\dots
\end{align}
\end{itemize}
\end{thm}
In particular, the theorem shows that a random triple in $[0,1]^3$ is
more likely than not to be cyclic, while only about 1\% of
all such triples are nontransitive.

In our second theorem, we consider the case of  general $n\ge4$. 
We derive upper and lower bounds for the probabilities $p_n$ that
show that $p_n$ converges to $1$ as $n\to\infty$.

\begin{thm}
\label{thm:main2}
For any integer $n\ge 4$ the probability $p_n$ that a random $n$-tuple 
$\xn\in[0,1]^n$ is cyclic satisfies
\begin{equation}
\label{eq:vol-cyclic-n-tuples}
	1-3\left(\frac{2}{\pi}\right)^n\le 
	p_n\le 1-2\left(\frac14\right)^n.
\end{equation}
In particular, $p_n$ converges exponentially to $1$ as $n\to\infty$.
\end{thm}

In our final result we consider the distribution of the smallest,
middle, and largest elements of a cyclic triple. Let 
$(X_1,X_2,X_3)$ be a random vector that is
uniformly distributed over the region
$\Cthree$ of all cyclic triples, and let $(X_1^*, X_2^*,X_3^*)$ be the
order statistics of $(X_1,X_2,X_3)$. Thus, $X_i$ is the $i$-th
coordinate of a random cyclic triple, while $X_i^*$ is the $i$-th
smallest among the three coordinates of a random cyclic triple, where 
``random'' is to be interpreted with respect to the usual Lebesgue
measure on $\Cthree$. 

In particular, $X_1^*$ is the smallest element of a random cyclic
triple, and the triple $(X_1,X_2,X_3)$  is nontransitive if and only
if $X_1^*>1/2$.  Thus it is of interest to determine the precise
distribution of the random variable $X_1^*$. More generally, in the
following theorem we determine the density function $f_i(x)$ for each of
the three random variables $X_i^*$, $i=1,2,3$.

\begin{thm}
\label{thm:main3}
Let $f_1(x)$, $f_2(x)$, and $f_3(x)$ denote, respectively,
the density function of the smallest, middle, and largest element 
of a random cyclic triple; i.e., $f_i(x)$ is the density of the random
variable $X_i^*$ defined above.  Then: 
\begin{align}
\label{eq:min-density}
    f_1(x)&=
    \begin{cases}
    \frac 3 {p_3} \left( x^3-3x^2 + \frac {1-x} {2-x} - (1-x) \ln (1-x)
    \right)
    &\text{if $0 \le x \le \frac{3-\sqrt 5}{2}$,}
    \\[.5ex]
    \frac 3 {p_3} \left( x^2-3x+1 - (1-x)\ln (1-x) \right)
    &\text{if 
    $\frac{3-\sqrt 5}{2} <x \le \frac12$,}
    \\[.5ex]
    \frac 3 {p_3} \left( x^2+x-1+(1-x) \ln (1-x)-2(1-x)\ln x \right)
    &\text{if $\frac12 < x \le \frac{\sqrt 5 - 1}2$,}
	\\[.5ex]
	0&\text{otherwise;}
    \end{cases}
    \displaybreak[3]
    \\[1ex]
    \label{eq:middle-density}
    f_2(x)&= 
    \begin{cases}
        \frac 3 {p_3} \left( 3x^2-x^3 \right)
	&\text{if $ 0 \le x \le \frac{3-\sqrt 5}2$,}
	\\[.5ex]
        \frac 6 {p_3} \left( 3x-x^2-\frac 1 {2(1-x)} \right)
	&\text{if $\frac{3-\sqrt 5}2 < x \le \frac12$,}
	\\[.5ex]
	f_2(1-x) &\text{if $\frac12 < x \le 1$,}
	\\[.5ex]
	0&\text{otherwise;}
    \end{cases}
    \\[1ex]
    \label{eq:max-density}
    f_3(x)&= f_1(1-x),
\end{align}
where $p_3$ is given by \eqref{eq:vol-cyclic3tuples}.
\end{thm}

\begin{figure}[H]
   \begin{center}
    \includegraphics[width = .45 \textwidth]{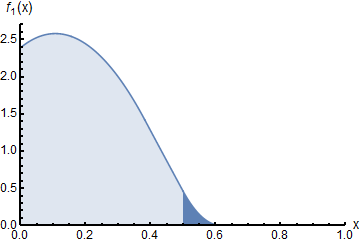}
    \hspace{1em}
    \includegraphics[width = .45\textwidth]{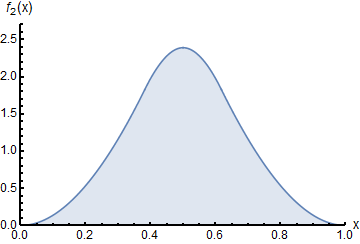}
\end{center}
\caption{The density functions $f_1(x)$ and $f_2(x)$ of the  
smallest,  respectively middle, value in a random cyclic triple. 
The dark-shaded
portion of the graph on the left corresponds to nontransitive triples.}
\label{fig:density-plots} 
\end{figure}

The densities $f_1(x)$ and $f_2(x)$ are shown
in Figure \ref{fig:density-plots}.
The dark-shaded portion of the graph
of $f_1(x)$ is the portion of this density corresponding to
nontransitive triples. As can be seen from the graph, the contribution 
of such triples to cyclic triples is quite small. This is consistent
with the result of Theorem \ref{thm:main1}, which shows that
nontransitive triples make up a proportion of only around
$0.627575/0.011217\approx 1/55$ of all cyclic triples. 

Formula \eqref{eq:min-density} shows that the  density $f_1(x)$
is supported on the interval $[0,(\sqrt{5}-1)/2]$. Note that 
the right endpoint of this interval, $(\sqrt{5}-1)/2$, is equal to the
number $\pi_3$  defined above (see \eqref{eq:pi-n-def} and
\eqref{eq:pi-n-initial-values}).  Figure \ref{fig:density-plots} shows
that the density $f_1(x)$ is strictly positive on the entire 
interval $[0,(\sqrt{5}-1)/2)$, and that it has a unique mode 
(i.e., local maximum) located at around $0.1$.

It is interesting to compare the distribution of $X_1^*$, the smallest element in a random \emph{cyclic} triple in $[0,1]^3$, to that of $X^*$, the smallest element in an \emph{unrestricted} random triple in $[0,1]^3$. An easy calculation shows that $X^*$ has density function  $f(x)=3(1-x)^2$, 
mean $1/4$, and median $1-2^{-1/3}=0.206\dots$. In contrast to the density function $f_1(x)$, the latter density function is supported on the entire interval $[0,1]$
and is strictly decreasing on this interval.  Further statistics on the distributions $f_1(x)$ and $f(x)$ are given in Table \ref{table:X1-statistics}. 
\begin{table}[ht]
    \centering
    \renewcommand{\arraystretch}{1.5}
    \addtolength{\tabcolsep}{3pt}
    \begin{tabular}{|c||c|c|c|}
        \hline
         Random Variable &Expected Value & Median & Mode \\
         \hline\hline
         $X_1^*$ &
         $0.211 \dots$ & $0.197 \dots$ & $0.107 \dots$ \\
         \hline
         $X^*$ &
         $0.25 \qquad$ & $0.206\dots$ & 0
         \\
         \hline
    \end{tabular}
    \caption{Statistics on the distributions of $X_1^*$, the smallest element
    in a random cyclic triple, and $X^*$, the smallest element 
    in a random triple in $[0,1]^3$.}
    \label{table:X1-statistics}
\end{table}

The graph on the right of Figure \ref{fig:density-plots} shows
the density of the middle value in a random cyclic triple.
As can be seen from formula \eqref{eq:middle-density},  
this distribution is symmetric with respect to the line $x=1/2$, and it
is supported on the full interval $[0,1]$.

Formula \eqref{eq:max-density} shows that the distribution of the 
largest value in a random cyclic triple, up to a reflection at the line
$x=1/2$, is the same as the distribution of the smallest value. 
This is a consequence of the symmetry properties of cyclic triples 
(cf. Lemma \ref{lem:symmetry-ntuples} below).

\bigskip

The remainder of this paper is organized as follows.  In Section
\ref{sec:auxiliary-results} we prove some elementary properties of
cyclic $n$-tuples, we present Trybula's characterization of cyclic
triples, and we derive a simplified form of this characterization
under additional assumptions on the triples.  Sections
\ref{sec:proof-theorem1}--\ref{sec:proof-theorem3} contain the proofs of
our main results, Theorems \ref{thm:main1}--\ref{thm:main3}.  We
conclude in Section \ref{sec:concluding-remarks} with a discussion of some
related questions and open problems suggested by our results.

\section{Auxiliary results}
\label{sec:auxiliary-results}

We begin by deriving some elementary properties of cyclic $n$-tuples.
Here, and in the remainder of the paper, we make the convention that 
subscripts of $n$-tuples are to be interpreted modulo $n$.
Thus, for example, the definition \eqref{eq:cyclic-tuples} of a 
cyclic $n$-tuple $(x_1,\dots,x_n)$ can be written more concisely as 
\[
	P(U_{i+1}>U_i)=x_i\quad (i=1,\dots,n).
\]
Given a real number $t\in [0,1]$, we write
\begin{equation}
\label{def:xbar}
\ov{t}=1-t.
\end{equation}

\begin{lem}
\label{lem:symmetry-ntuples}
Let $\xn\in[0,1]^n$ be a cyclic $n$-tuple.  Then:
\begin{itemize}
\item[(i)] Any cyclic permutation of $\xn$ is cyclic; that is, for any
        $i\in\{1,2,\dots,n\}$, the tuple $(x_i,x_{i+1},\dots,x_{i+n})$
                (with subscripts interpreted modulo $n$) is cyclic as well.
        \item[(ii)] The ``reverse'' tuple
                $(x_n,x_{n-1},\dots,x_1)$ is cyclic.
\item[(iii)] The ``complementary'' tuple $\ovxn$ is cyclic.
\end{itemize}
\end{lem}

\begin{proof}
Part (i) of the lemma follows
immediately from the definition \eqref{eq:cyclic-tuples}
of cyclic $n$-tuples.

For parts (ii) and (iii) suppose $\xn\in[0,1]^n$ is a
cyclic $n$-tuple with associated
        random variables $U_1,\dots,U_n$ satisfying \eqref{eq:cyclic-tuples}.
Setting $U_i^*=-U_{n+1-i}$ we have
\begin{align*}
		P(U_{i+1}^*> U_i^*)&=P(-U_{n-i}> -U_{n+1-i})
\\
		&=P(U_{n+1-i}> U_{n-i})=x_{n-i}
	\quad (i=0,1,\dots,n-1),
\end{align*}
which shows that the tuple $(x_n,x_{n-1},\dots,x_1)$ is cyclic.

Similarly, the fact that $\ovxn$ is cyclic follows by
        letting $U_i^*=-U_{i}$ and observing that
\begin{align*}
	P(U_{i+1}^*> U_i^*)&=P(-U_{i+1}> -U_{i})
	\\
	&=P(U_{i}> U_{i+1})=1-x_{i}
	\quad (i=1,\dots,n).
\end{align*}
This completes the proof of the lemma.
\end{proof}

\begin{lem}
        \label{lem:symmetry-triples}
If $\xyz\in [0,1]^3$ is a cyclic triple,
then so is any permutation of $\xyz$.
\end{lem}
\begin{proof}
Suppose $\xyz\in[0,1]^3$ is cyclic. By part (i) of Lemma
        \ref{lem:symmetry-ntuples} the \emph{cyclic} permutations $(y,z,x)$ and
        $(z,x,y)$ are also cyclic. By part (ii) the reverse triple $(z,y,x)$
        is cyclic. Applying part (i) again to the triple $(z,y,x)$,
        we obtain that $(y,x,z)$ and $(x,z,y)$ are cyclic as well. Hence
all permutations of $\xyz$ are cyclic.
\end{proof}

The next result contains Trybula's characterization of cyclic triples
\cite{trybula1961}.   We state this characterization in the slightly
different---though equivalent---version given by Suck \cite{suck2002}.

\begin{lem}[Trybula {\cite[Theorem 1]{trybula1961}}; 
Suck {\cite[Theorems 2 and 3]{suck2002}}]
\label{lem:trybula-characterization}
	A triple $\xyz\in [0,1]^3$ is cyclic if and only if it
 satisfies the following inequalities:
 \begin{align}
\label{eq:cyclic-bound1}
	 \min\left(x+yz, y+zx, z+xy\right)\le 1,
\\
\label{eq:cyclic-bound2}
	 \min\left(\ovx+\ovy\;\ovz,
\ovy+\ovz\;\ovx,
	 \ovz+\ovx\;\ovy\right)\le 1.
\end{align}
\end{lem}

The conditions \eqref{eq:cyclic-bound1} and \eqref{eq:cyclic-bound2}
in this characterization are rather unwieldy to work with directly.
However, by imposing additional constraints on the variables $x$, $y$,
$z$, the conditions simplify significantly as the next lemma shows.

\begin{lem}
	\label{lem:cyclic-characterizations}
	Let $\xyz\in[0,1]^3$.
	\begin{itemize}
		\item[(i)] 
			If $x\le y\le z$, then the triple
			$\xyz$ is cyclic if and only if 
			it satisfies the following two inequalities:
			\begin{align}
			\label{eq:cyclic-bound-simplified1}
				x+yz&\le 1,
				\\
			\label{eq:cyclic-bound-simplified2}
				\ovz+\ovx\, \ovy&\le 1.
			\end{align}
		\item[(ii)] 
			If $x=\min(x,y,z)$ and $y,z\ge 1/2$,
			then $\xyz$ is cyclic
			if and only if \eqref{eq:cyclic-bound-simplified1}
			holds.
	\end{itemize}
\end{lem}

\begin{proof}
	We need to show that the conditions \eqref{eq:cyclic-bound1} and
	\eqref{eq:cyclic-bound2} reduce to  
		\eqref{eq:cyclic-bound-simplified1}
			and \eqref{eq:cyclic-bound-simplified2} 
	under the assumptions of part (i), and to 
	\eqref{eq:cyclic-bound-simplified1}
	under the assumptions of part (ii).

	Suppose first that $\xyz$ satisfies the conditions of part (i), 
	i.e., that $x\le y\le z$. Then 
$x(1-y)\le z(1-y)$ and $x(1-z)\le y(1-z)$, and therefore
	$x+yz\le z+xy$ and $x+yz\le y+zx$. Thus, the minimum 
	on the left of the inequality
	\eqref{eq:cyclic-bound1} is equal to $x+yz$,
	and the inequality therefore holds if and only if $x+yz\le 1$, 
	Similarly, noting that $x\le y\le z$ is equivalent to 
	$\ovz\le \ovy\le \ovx$, we see that
	condition \eqref{eq:cyclic-bound2} is 
	equivalent to $\ovz + \ovx\, \ovy\le 1$.
	This proves part (i) of the lemma.

	Next, suppose that $y,z\ge 1/2$ and $x=\min(x,y,z)$. Then 
	either $z=\max(x,y,z)$ or $y=\max(x,y,z)$.
	In the first case we have 
	$x\le y\le z$, so part (i) of the lemma  applies and
	shows that the triple $\xyz$ is cyclic if and only if it satisfies
		\eqref{eq:cyclic-bound-simplified1} and
		\eqref{eq:cyclic-bound-simplified2}.
		But since $y,z\ge 1/2$,
		we have $\ovy,\ovz\le 1/2$ and
		therefore $\ovz + \ovx\, \ovy\le (1/2)+ 1\cdot (1/2)=1$.
	Hence condition \eqref{eq:cyclic-bound-simplified2} holds
	trivially, so $\xyz$  is cyclic if and only if
	\eqref{eq:cyclic-bound-simplified1} holds.

    In the case where we have $x \le z \le y$, we note that, by Lemma \ref{lem:symmetry-triples}, $(x,y,z)$ is cyclic if and only if $(x,z,y)$ is cyclic. Applying the above argument to $(x,z,y)$ then yields the same conclusion. This completes the proof of part (ii).
\end{proof}

\section{Proof of Theorem~\ref{thm:main1}}
\label{sec:proof-theorem1}

By \eqref{eq:pn-definition} we have $p_3=\vol(\Cthree)$ and
$p_3^*=\vol(\NTthree)$, so  computing these probabilities amounts
to computing the volumes of the regions $\Cthree$ and $\NTthree$ of
cyclic and nontransitive triples.

We begin by using the symmetry properties established in Lemmas
\ref{lem:symmetry-ntuples} and \ref{lem:symmetry-triples} to reduce 
this computation to one involving simpler regions.  Let
\begin{align}
\label{eq:def-CthreeI}
        \CthreeI &=\{(x,y,z)\in \Cthree: 1/2<x\le 1\text{ and }x\le y,z \le 1\},
\\
\label{eq:def-CthreeII}
        \CthreeII &=\{(x,y,z)\in \Cthree: 0\le x<1/2\text{ and }
        1/2< y,z \le 1\}.
\end{align}

\begin{lem}
\label{lem:volume3-splitup}
We have
\begin{align}
	\label{eq:volumeI-splitup}
        \vol(\NTthree) &= 3\vol(\CthreeI),
\\
	\label{eq:volume-splitup}
        \vol(\Cthree) &= 6\vol(\CthreeI) + 6\vol(\CthreeII).
\end{align}
\end{lem}

\begin{proof}
By definition, the set $\NTthree$ of nontransitive triples consists of those
cyclic triples $\xyz$ in which all coordinates are $>1/2$.  Since, by Lemma
\ref{lem:symmetry-triples}, the ``cyclic triple'' 
property is invariant with respect
to taking permutations, the volume of this set is three times the volume of
the set of those triples $\xyz$ in $\NTthree$ which satisfy $x=\min\xyz$, i.e., the set $\CthreeI$.  This proves \eqref{eq:volumeI-splitup}.

To prove \eqref{eq:volume-splitup}, note first that we may ignore 
triples $\xyz$ in which one of the coordinates is equal to $1/2$ as these do
not contribute to the volume.  We classify the remaining triples $\xyz$
	in $\Cthree$ into $8$ mutually disjoint classes, according to their
	\emph{signature} $\sigma$, defined as 
\begin{equation*}
	\sigma
	=(\sgn(x-1/2),\sgn(y-1/2),\sgn(z-1/2)),
\end{equation*}
where $\sgn(t)=1$ if $t>0$  and $\sgn(t)=-1$ if $t<0$.
For example, a triple $\xyz$ with $x<1/2$, $y>1/2$, $z>1/2$ has signature
$(-1,1,1)$.

	Letting $\Cthree^{\,\sigma}$ denote the set of cyclic triples $\xyz$
with signature $\sigma$, we then have
	\begin{equation}
		\label{eq:volume-splitup-proof1}
		\vol(\Cthree) = \sum_{\sigma=(\pm1,\pm1,\pm1)}
		\vol(\Cthree^{\,\sigma}),
	\end{equation}
where the sum is over all $8$ possible values of the signature $\sigma$.
Now note that the cyclic 
triples with signature $(1,1,1)$ are exactly the nontransitive triples, 
and the cyclic triples
with signature $(-1,1,1)$ are exactly those counted in the set $\CthreeII$. Thus we have
\begin{align}
		\label{eq:volume-splitup-proof2a}
		\vol(\Cthree^{(1,1,1)})&=\vol(\NTthree)=3\vol(\CthreeI),
\\
		\label{eq:volume-splitup-proof2b}
		\vol(\Cthree^{(-1,1,1)})&=\vol(\CthreeII).
\end{align}

Next, observe that if $\xyz$ has signature $\sigma$, then 
the complementary triple $\ovxyz=(1-x,1-y,1-z)$ has
signature $-\sigma$. Since, by Lemma \ref{lem:symmetry-ntuples}(iii), a
triple $\xyz$ is cyclic if and only if $\ovxyz$ is cyclic, it follows that
\begin{equation} 
\label{eq:volume-splitup-proof3}
\vol(\Cthree^{(-1,-1,-1)})= \vol(\Cthree^{(1,1,1)})= \vol(\NTthree).
\end{equation}
Finally, using Lemma \ref{lem:symmetry-ntuples}(iii) along with Lemma
\ref{lem:symmetry-triples}, we see that
\begin{align}
\label{eq:volume-splitup-proof4}
	\vol(\Cthree^{(-1,-1,1)})&= 
	\vol(\Cthree^{(1,-1,-1)})=
\vol(\Cthree^{(-1,1,-1)})
\\
\notag
	&=\vol(\Cthree^{(1,1,-1)})
	=\vol(\Cthree^{(1,-1,1)})
	=\vol(\Cthree^{(-1,1,1)})
	\\
	\notag
	&=\vol(\CthreeII).
\end{align}
	Combining \eqref{eq:volume-splitup-proof1}--\eqref{eq:volume-splitup-proof4} yields \eqref{eq:volume-splitup}.
\end{proof}

Let 
\begin{equation}
	\label{eq:omega-def}
	\omega=\frac{\sqrt{5}-1}{2}=0.618034\dots
\end{equation}
and note that $\omega$ is the positive root of the quadratic equation
\begin{equation}
	\label{eq:omega-relation}
	\omega^2+\omega=1.
\end{equation}
We remark that $\omega$ is also equal to the number $\pi_3$
defined in \eqref{eq:pi-n-def} and \eqref{eq:pi-n-initial-values}, i.e.,
$\omega$ is the largest number for which there exists a cyclic
triple $\xyz$ with $\min\xyz\ge \omega$. We will, however, not use 
this fact in our proof. 

\begin{lem}
	\label{lem:CthreeI-II-characterization}
\mbox{}
\begin{itemize}
\item[(i)]
A triple $\xyz$ belongs to the set $\CthreeI$ if and only if it satisfies
\begin{equation}
\label{eq:CthreeI-characterization}
\left\{
\begin{aligned}
	&\frac12<x\le \omega
	\\
	&x\le y\le \frac{1-x}{x}
\\
	&x\le z\le \frac{1-x}{y}
\end{aligned}
\right\}.
\end{equation}
\item[(ii)]
A triple $\xyz$ belongs to the set $\CthreeII$ if and only if it satisfies
\begin{equation}
\label{eq:CthreeII-characterization}
\left\{
\begin{aligned}
		&0\le x<\frac12
		\\
		&\frac12< y\le 1-x
		\\
		&\frac12< z\le 1
	\end{aligned}
\right\}
\quad\text{or}\quad
\left\{
\begin{aligned}
		&0\le x<\frac12
		\\
		&1-x<y\le 1
		\\
		&\frac12<z\le \frac{1-x}{y}
\end{aligned}
\right\}.
\end{equation}
\end{itemize}
\end{lem}

\begin{proof}
By definition (see \eqref{eq:def-CthreeI} and \eqref{eq:def-CthreeII})
the sets $\CthreeI$ and $\CthreeII$
consist of those cyclic
triples $\xyz\in[0,1]^3$
that satisfy $1/2< x\le y,z\le 1$ and $0\le x<1/2< y,z\le 1$,
respectively.    By part (ii) of Lemma
\ref{lem:cyclic-characterizations}, under either of the latter two
conditions, a triple $\xyz$ is cyclic if and only if it satisfies
$x+yz\le 1$. 

Next, note that any triple $\xyz\in \CthreeI$ must satisfy   $x\le y$ and
$x\le z$, 
so the inequality $x+yz\le 1$ can only hold if $x+x^2\le 1$, i.e., if 
$x\le \omega$, where $\omega$ is defined by
 \eqref{eq:omega-def} and \eqref{eq:omega-relation}.
For each
$x$ in the range $1/2<x\le \omega$, the set of pairs $(y,z)$ with $x\le
y,z\le 1$ satisfying $x+yz\le 1$ is nonempty and consists of exactly
those pairs that satisfy $x\le y\le (1-x)/x$ and 
$x\le z\le (1-x)/y$.  It follows
that a triple $\xyz$ belongs to $\CthreeI$ if and only if it satisfies
\eqref{eq:CthreeI-characterization}.  This proves part (i) of the lemma.

For the proof of part (ii), note that $\CthreeII$ is the set of triples
$\xyz$ satisfying $0\le x<1/2$, $1/2<y,z\le 1$, and $x+yz\le 1$.
For any fixed pair $(x,y)$ with $0\le x<1/2$ and $1/2<y\le 1$,
the set of values $z$ with $1/2<z\le 1$ for which $x+yz\le 1$ holds
is exactly the interval $(1/2, \min((1-x)/y,1)]$. 
The latter interval is nonempty, and it 
reduces to $(1/2,1]$ if $y\le 1-x$, and to 
$(1/2,(1-x)/y]$ if $1-x<y\le 1$.  The desired 
characterization \eqref{eq:CthreeII-characterization} now follows.
\end{proof}

\begin{lem}
\label{lem:vol-CthreeII}
We have
\begin{equation}
\label{eq:vol-CthreeII}
\vol(\CthreeII)= \frac 3 {16} - \frac {\ln 2} 8.
\end{equation}
\end{lem}

\begin{proof}
Using the characterization 
\eqref{eq:CthreeII-characterization} of the set $\CthreeII$ we obtain
    \begin{align*}
        \vol(\CthreeII) 
&= \int_{0}^{1/2} \int_{1/2}^{1-x} \int_{1/2}^{1} \,
dz\, dy\, dx 
+ \int_{0}^{1/2} \int_{1-x}^1 \int_{1/2}^{\frac {1-x} y}\,
dz\,dy\,dx
\\
       &= \frac12 \int_0^{1/2} \left(\frac12- x\right) \, dx
       + \int_{0}^{1/2} \int_{1-x}^1 \left(\frac {1-x} y 
       - \frac12\right) \, dy\, dx
\\      
       &= \frac12 \int_0^{1/2} \left(\frac12- 2x\right) \, dx
	 - \int_0^{1/2}(1-x) \ln (1-x)\, dx 
\\
	 &=- \int_0^{1/2}(1-x) \ln (1-x)\, dx 
\\
        &= - \left[ \frac {(1-x)^2}{4} (1-2\ln(1-x))  
	\right]_0^{1/2}
        =  \frac 3 {16} - \frac {\ln 2} 8.
\qedhere
\end{align*}
\end{proof}

\begin{lem}
\label{lem:vol-CthreeI}
We have
\begin{equation}
\label{eq:vol-CthreeI}
\vol(\CthreeI)= 
\frac{\ln 2}{8} -\ln (2\omega)+\frac{11\omega}{12}-\frac{7}{16},
\end{equation}
where $\omega=(\sqrt{5}-1)/2$ is defined as in \eqref{eq:omega-def}.
\end{lem}

\begin{proof}
Using the characterization \eqref{eq:CthreeI-characterization}
of the set $\CthreeI$ we obtain, on noting that
$1>(1-x)/x\ge x$ for $1/2<x\le \omega$ (since $\omega$ is the
positive root of $\omega^2=1-\omega$) 
and $1>(1-x)/y\ge x $ for $x\le y\le (1-x)/x$,
\begin{align}
\label{eq:volI-integral0}
        \vol(\CthreeI) 
	&= \int_{1/2}^{\omega} \int_{x}^{\frac {1-x} x} 
	\int_x^{\frac {1-x} y} dz\, dy\, dx
        = \int_{1/2}^{\omega} \int_{x}^{\frac {1-x} x}
	\left(\frac {1-x} y - x\right)\, dy\, dx 
	\\
	\notag
        &= 
	\int_{1/2}^{\omega} (1-x) 
	\ln \left( \frac {1-x}{x^2}\right) \, dx -
	\int_{1/2}^{\omega} x\left(\frac {1-x} x - x\right)\, dx 
	\\
	\notag
	&= I_1-I_2,
\end{align}
say.  The integrals $I_1$ and $I_2$ can be evaluated as follows,
using the relations (see \eqref{eq:omega-relation}) 
$\omega^2=1-\omega$ and $\omega^3=\omega-\omega^2=2\omega-1$:
\begin{align}
    \label{eq:volI-integral1}
I_1 &= \left[\frac{-(1-x)^2}{2}
\ln\left(\frac{1-x}{x^2}\right)
\right]_{1/2}^\omega +
\int_{1/2}^\omega \frac{(1-x)^2}{2}\left(-\frac1{1-x}-\frac{2}{x}\right)
\, dx
\\
\notag
&= 
\frac{-(1-\omega)^2}{2}\ln\left(\frac{1-\omega}
{\omega^2}\right)
+\frac{\ln 2}{8}
+\int_{1/2}^\omega\left(-\frac1x+\frac32-\frac{x}2\right)\, dx
\\
\notag
 &=\frac{\ln 2}{8}
 -\ln (2\omega) +\frac{3(\omega-1/2)}{2}
-\frac{\omega^2-1/4}{4}
\\
\notag
 &=\frac{\ln 2}{8}
 -\ln\left(2\omega\right)+\frac{7\omega}{4}-\frac{15}{16},
\\[1ex]
\label{eq:volI-integral2}
I_2 &= \int_{1/2}^{\omega} 
	\left(1-x-x^2\right)\, dx 
	= \omega-\frac12
	-\frac{\omega^2-1/4} 2
	-\frac{\omega^3-1/8} 3 
=\frac{5\omega}{6}-\frac12.
\end{align}
Substituting \eqref{eq:volI-integral1} and \eqref{eq:volI-integral2}
into \eqref{eq:volI-integral0} yields \eqref{eq:vol-CthreeI}.
\end{proof}

\begin{proof}[Proof of Theorem \ref{thm:main1}]
Combining \eqref{eq:pn-definition} and 
Lemmas \ref{lem:volume3-splitup}, 
\ref{lem:vol-CthreeII}, and 
\ref{lem:vol-CthreeI},  we obtain
\begin{align*}
       p_3^* &= \vol(\NTthree) = 3 \vol (\CthreeI)
\\
       &= 3 
\left(\frac{\ln 2}{8}
-\ln (2\omega)+\frac{11\omega}{12}-\frac{7}{16}\right)
\\
&= 3\left(\frac{\ln 2}{8}
-\ln(\sqrt{5}-1)+\frac{11\sqrt{5}-11}{24}-\frac{7}{16}\right)
\\
&= \frac {3 \ln 2} 8 - 3 \ln(\sqrt 5 - 1) + \frac {11 \sqrt 5} {8}
	- \frac {43} {16}
	=0.01121759\dots
\end{align*}
and 
\begin{align*}
        p_3&=\vol(\Cthree) = 
	6\vol(\CthreeI) + 6\vol(\CthreeII) \\
&= 6\left(\frac{\ln 2}{8}
-\ln(\sqrt{5}-1)+\frac{11\sqrt{5}-11}{24}-\frac{7}{16}\right) 
+
	6 \left(\frac 3 {16} - \frac {\ln 2} 8\right)
\\
        &= \frac {11 \sqrt 5 - 17} 4 - 6\ln (\sqrt 5 - 1)
	=0.6275748\dots
\end{align*}
These are the desired formulas 
\eqref{eq:vol-nontransitive3tuples} 
and 
\eqref{eq:vol-cyclic3tuples}. 
\end{proof}

\section{Proof of Theorem~\ref{thm:main2}}
\label{sec:proof-theorem2}

\begin{proof}[Proof of Theorem \ref{thm:main2}, upper bound] 
For the upper bound in \eqref{eq:vol-cyclic-n-tuples},
 note
that, by \eqref{eq:pi-n-general-formula}, a cyclic tuple $\xn\in[0,1]^n$ must satisfy
\begin{equation*}
  \min\xn\le \pi_n=1-\frac{1}{4\cos^2(\pi/(n+2))} <\frac34.
  \end{equation*}
Therefore any tuple $\xn\in[3/4,1]^n$  is \emph{not} cyclic. Moreover, since, by Lemma \ref{lem:symmetry-ntuples}(iii),
a tuple $\xn$ is cyclic if and only if the complementary tuple
$(1-x_1,\dots, 1-x_n)$ is cyclic, any tuple $\xn$ satisfying $x_i\in[0,1/4]$ for all $i$ is also \emph{not} cyclic. Thus the set $\Cn$  of cyclic tuples lies in the complement of the sets $[3/4,1]^n$ and $[0,1/4]^n$. Hence we have
\begin{equation*}
p_n=\vol(\Cn)\le 1-\vol\left([3/4,1]^n\right)-\vol\left([0,1/4]^n\right)
=1-2\left(\frac14\right)^n,
\end{equation*}
which is the desired upper bound.
\end{proof}

We next turn to the lower bound in \eqref{eq:vol-cyclic-n-tuples}.  The
argument is based on the following lemma, which gives a sufficient
condition for an $n$-tuple to be cyclic.  Recall our convention that
subscripts in $n$-tuples $\xn$ are to be interpreted modulo $n$.

\begin{lem}
\label{lem:up-down-construction}
Let $\xn \in [0,1]^n$, and suppose that there exists 
an index $i\in\{1,\dots,n\}$ such that
\begin{equation}
    \label{eq:pairwise-sum-condition3}
         x_i+x_{i+1} \ge 1 \text{ and } x_{i+2}+x_{i+3} \le 1.
   \end{equation}
Then $\xn$ is cyclic.
\end{lem}

\begin{proof}
Since, by Lemma \ref{lem:symmetry-ntuples}(i), a cyclic permutation of a
cyclic tuple is also cyclic,  we may assume without loss of generality  
that the assumption \eqref{eq:pairwise-sum-condition3} of the lemma 
holds with $i=n-2$, i.e., that
\begin{equation}
\label{eq:pairwise-sum-condition3a}
        x_{n-2}+x_{n-1} \ge 1\text{ and } x_n+x_1 \le 1.
\end{equation}

Let $\xn\in[0,1]^n$ be an $n$-tuple satisfying
\eqref{eq:pairwise-sum-condition3a}. Define
independent random variables $U_1, \dots, U_n$ 
with values in $\{-n+1,-n+2,\dots,n+1,n+2\}$ as follows:
\begin{equation}
\label{eq:Ui-construction}
\left\{
    \begin{aligned}
        P(U_{n-1} = -n+1) &= 1-x_{n-2}, \\
        P(U_i = -i) &= 1-x_{i-1} \quad (n-2 \ge i \ge 3), \\
        P(U_2 = -2) &= 1- \frac {x_1} {1-x_n}, \\
        P(U_1 = 0) &= 1-x_n, \\
        P(U_2 = 2) &= \frac {x_1} {1-x_n}, \\
        P(U_i = i) &= x_{i-1} \quad  (3 \le i \le n-2), \\
        P(U_{n-1} = n-1) &= x_{n-1}+x_{n-2}-1, \\
        P(U_n = n) &= 1, \\
        P(U_1 = n+1) &= x_n, \\
        P(U_{n-1} = n+2) &= 1-x_{n-1}.
    \end{aligned}
\right.
\end{equation}
The values of the random variables $U_i$ are shown below:
\begin{center}
\usetikzlibrary{arrows}
\begin{tikzpicture}
    \draw[latex-latex] (-7.5,0) -- (6.5,0) ;
    \draw[black] (-7,3pt) -- (-7,-3pt) node[below] {$U_{n-1}$} ;
    \node at (-7, 11pt) {$-(n-1)$};
    \node[below] at (-6,-6pt) {\dots} ;
    \draw[black] (-5,3pt) -- (-5,-3pt) node[below] {$U_i$} ;
    \node at (-5, 11pt) {$-i$};
    \node[below] at (-4,-6pt) {\dots} ;
    \draw[black] (-3,3pt) -- (-3,-3pt) node[below] {$U_2$} ;
    \node at (-3, 11pt) {$-2$};
    \draw[black] (-2,3pt) -- (-2,-3pt) node[below] {$U_1$} ;
    \node at (-2, 11.5pt) {$0$};
    \draw[black] (-1,3pt) -- (-1,-3pt) node[below] {$U_2$} ;
    \node at (-1, 11.5pt) {$2$};
    \node[below] at (0,-6pt) {\dots} ;
    \draw[black] (1,3pt) -- (1,-3pt) node[below] {$U_i$} ;
    \node at (1, 11.5pt) {$i$};
    \node[below] at (2,-6pt) {\dots} ;
    \draw[black] (3,3pt) -- (3,-3pt) node[below] {$U_{n-1}$} ;
    \node at (3, 11.5pt) {$n-1$};
    \draw[black] (4,3pt) -- (4,-3pt) node[below] {$U_n$} ;
    \node at (4, 11pt) {$n$};
    \draw[black] (4.9,3pt) -- (4.9,-3pt) node[below] {$U_1$} ;
    \node at (4.9, 11.5pt) {$n+1$};
    \draw[black] (6,3pt) -- (6,-3pt) node[below] {$U_{n-1}$} ;
    \node at (6, 11.5pt) {$n+2$};
\end{tikzpicture}
\end{center}

The idea behind this construction is the following: If we let $U_i$ be
random variables with values $\pm i$, then $U_{i+1}>U_i$ holds if and
only if $U_{i+1}=i+1$.  Thus, if we require that  $P(U_{i+1}=i+1)=x_i$
(and hence $P(U_{i+1}=-i-1)=1-x_i$), then $U_{i+1}>U_i$ holds with the
desired probability $x_i$.  This is indeed how the variables  
$U_i$ in \eqref{eq:Ui-construction} are defined for the ``interior''
indices $i=3,\dots,n-2$,  so the probabilities $P(U_{i+1}>U_i)$
have the desired value $x_i$ if $2\le i\le n-3$.

For the remaining indices $i=1,2,n-1,n$, the definition of $U_i$ 
has to be adjusted to ensure that the ``wrap-around'' probability $P(U_n>U_1)$
also has the desired value. The following calculations show 
that if $U_1, U_2, U_{n-1}, U_n$ are defined
as in \eqref{eq:Ui-construction}, then the remaining
probabilities $P(U_{i+i}>U_i)$ also have the desired values:  
\begin{align*}
P(U_2>U_1) & = P(U_2=2)P(U_1=0)=\frac{x_1}{1-x_n}\cdot (1-x_n)=x_1,
\\
P(U_{n}>U_{n-1})&=1-P(U_{n-1}=n+2)=1-(1-x_{n-1})=x_{n-1},
\\
P(U_{n-1}>U_{n-2})&=1-P(U_{n-1}=-n+1)=1-(1-x_{n-2})=x_{n-2},
\\
P(U_1>U_n) & = P(U_1=n+1)=x_n.
\end{align*}
The assumption \eqref{eq:pairwise-sum-condition3a} ensures that the
numbers $x_1/(1-x_n)$ and $x_{n-1}+x_{n-2}-1$ arising in the definition
of $U_2$ and $U_{n-1}$ are contained in the interval $[0,1]$ and thus
can represent probabilities.

Thus the tuple $\xn$ is indeed cyclic.
\end{proof}

\begin{lem}
\label{lem:non-cyclic-condition}
Suppose $\xn\in[0,1]^n$ satisfies \textbf{neither} of the conditions
\begin{equation}
\label{eq:pairwise-sum-condition1}
        x_i+x_{i+1}<1
\quad (i=1,\dots,n)
\end{equation}
and
\begin{equation}
\label{eq:pairwise-sum-condition2}
        x_i+x_{i+1}>1
\quad (i=1,\dots,n).
\end{equation}
Then $\xn$ is cyclic.
\end{lem}

\begin{proof}
Let 
\[
s_i=x_i+x_{i+1}.
\]
By Lemma \ref{lem:up-down-construction}, it suffices to show that 
if neither \eqref{eq:pairwise-sum-condition1} nor
\eqref{eq:pairwise-sum-condition2} hold, 
then there exists an index $i\in\{1,\dots,n\}$ such that
\begin{equation}
    \label{eq:pairwise-sum-condition3b}
         s_i \ge 1 \text{ and } s_{i+2}\le 1.
   \end{equation}
We split the argument into two cases based on the parity of $n$.

If $n$ is odd,  consider the sequence of $n$ numbers 
$s_1, s_3, \dots, s_{n}, s_2,s_4,\dots, s_{n-1}$.
If there is no $i$ for which \eqref{eq:pairwise-sum-condition3b} holds, 
then the numbers in this sequence must be either all strictly greater 
than $1$ or all strictly less than $1$.  But since this sequence is a
permutation of $s_1,s_2,\dots,s_n$, it then follows that either 
\eqref{eq:pairwise-sum-condition1} or \eqref{eq:pairwise-sum-condition2}
holds, contradicting the assumption of the lemma.

If $n$ is even, consider the two sequences 
$A = \{s_1,s_3,\dots, s_{n-1}\}$ and $B = \{s_2,s_4,\dots, s_{n}\}$. 
If there is no $i$ for which \eqref{eq:pairwise-sum-condition3b} holds,
we conclude as before that, \emph{within each of these two sequences}, 
either all elements are greater than $1$ or all elements are less than $1$.
If the elements of \emph{both} sequences are all greater than $1$ or 
all less than $1$, then \eqref{eq:pairwise-sum-condition1} or
\eqref{eq:pairwise-sum-condition2} follows, and we again have a
contradiction.

In the remaining case the  elements of one sequence
are all greater than $1$ and those of the other sequence are all
less than $1$.  Since each sequence has exactly $n/2$ elements,  
one of the two sums $\sum_{a\in A}a$ and $\sum_{b\in B}b$
must be strictly greater than $n/2$, while the other sum must 
be strictly less than $n/2$.  On the other hand, the identity 
\begin{align*}
\sum_{a \in A} a &= \sum_{j=1}^{n/2} (x_{2j-1}+x_{2j})
=\sum_{i=1}^n x_i = \sum_{j=1}^{n/2} (x_{2j}+x_{2j+1})
=\sum_{b \in B}b 
\end{align*}
shows that the two sums $\sum_{a\in A}a$ and $\sum_{b\in B}b$
are in fact equal. Thus, this case cannot occur and the
proof of the lemma is complete.
\end{proof}

Define the sets  
\begin{align}
\label{eq:DnI-def}
\DnI&=\{\xn\in[0,1]^n: \text{ $\xn$ satisfies  
\eqref{eq:pairwise-sum-condition1}}\},
\\
\label{eq:DnII-def}
\DnII&=\{\xn\in[0,1]^n: \text{ $\xn$ satisfies  
\eqref{eq:pairwise-sum-condition2}} \},
\\
\label{eq:Dnstar-def}
\Dnstar&=\{\xn\in\DnI: \text{ $x_1=\min\xn$}\}.  
\end{align}

\begin{lem}
\label{lem:vol-Dn}
We have
\begin{align}
\label{eq:Dn-symmetry1}
\vol(\DnI)&=\vol(\DnII),
\\
\label{eq:Dn-symmetry2}
\vol(\DnI)&=n\vol(\Dnstar),
\\
\label{eq:Cn-lower-bound}
\vol(\Cn) &\ge 1-2n\vol(\Dnstar).
\end{align}
\end{lem}

\begin{proof} 
First note that an $n$-tuple $\xn$ satisfies condition 
\eqref{eq:pairwise-sum-condition2} if and only if the complementary
$n$-tuple $\ovxn=(1-x_1,\dots,1-x_n)$ satisfies 
\eqref{eq:pairwise-sum-condition1}.  The transformation $\xn\to\ovxn$
then shows that the regions $\DnI$ and $\DnII$ have the same volume.
This yields \eqref{eq:Dn-symmetry1}.

For the proof of \eqref{eq:Dn-symmetry2}, note that if a tuple $\xn$
satisfies condition \eqref{eq:pairwise-sum-condition1}, then so does any
cyclic permutation of this tuple. Therefore the set $\DnI$ is invariant
with respect to taking cyclic permutations.  It follows that the 
volume of $\DnI$ is $n$ times that of $\Dnstar$,
the set consisting of those tuples in $\DnI$ with smallest component
$x_1$. This proves \eqref{eq:Dn-symmetry2}.

Finally, \eqref{eq:Cn-lower-bound} follows on noting that, 
by Lemma \ref{lem:non-cyclic-condition}, the set $\Cn$ of cyclic
$n$-tuples contains the complement of the set $\DnI\cup\DnII$. Therefore, 
\begin{align*}
\vol(\Cn) &\ge 1-\vol(\DnI\cup\DnII)\ge 1-\vol(\DnI)-\vol(\DnII)
\\
&=1-2\vol(\DnI)=1-2n\vol(\Dnstar),
\end{align*}
where in the last step we have used \eqref{eq:Dn-symmetry1} and
\eqref{eq:Dn-symmetry2}.  
\end{proof}

To obtain the desired lower bound for the volume of $\Cn$, it now remains
to obtain an appropriate upper bound for the volume of the region
$\Dnstar$.  The following lemma expresses this volume in terms of a
combinatorial quantity counting \emph{alternating} permutations.  Here a
permutation of a set of $n$ distinct real numbers $x_1,\dots,x_n$ is
called \emph{up-down alternating} if it satisfies $x_1<x_2>x_3<\dots$,
i.e., if $x_{i+1}-x_i$ is positive for odd $i$ and negative for even $i$;
see Andr\'e \cite{andre1881} for more about such permutations.

\begin{lem}
\label{lem:vol-Dn-formula}
We have
\begin{equation}
\label{eq:vol-Dn-formula}
\vol(\Dnstar)
        =\frac{A_{n-1}}{(2n)(n-1)!},
\end{equation}
where $A_{n-1}$ is the number of up-down alternating permutations of
length $n-1$.
\end{lem}

\begin{proof}
By definition, $\Dnstar$ is the set of $n$-tuples $\xn\in[0,1]^n$
satisfying the inequalities \eqref{eq:pairwise-sum-condition1}
and in addition $x_1=\min\xn$.  Under the latter condition
the last inequality in \eqref{eq:pairwise-sum-condition1},
$x_{n}+x_1<1$,
is implied by the second-last inequality, $x_{n-1}+x_{n}<1$, and thus can be
omitted. Moreover, the inequalities \eqref{eq:pairwise-sum-condition1}
can only hold if $0\le x_1<1/2$. Thus, the set $\Dnstar$ is the set of
tuples $\xn$ satisfying
\begin{equation}
\label{eq:Dnstar-characterization}
\left\{
\begin{aligned}
        &0\le x_1<1/2
        \\
        &x_1\le x_2<1-x_1
\\
&\vdots
\\
        &x_1\le x_n<1-x_{n-1}
\end{aligned}
\right\}.
\end{equation}
Hence we have
\begin{align}
\label{eq:Dnstar-vol-calculation1}
   \vol(\Dnstar) &=\int_0^{1/2} \int_{x_1}^{1-x_1} \int_{x_1}^{1-x_2}
    \cdots \int_{x_1}^{1-x_{n-1}} dx_n 
    \cdots dx_3 dx_2 dx_1.
\end{align}
To evaluate the latter integral, we apply the change of variables 
\begin{equation}
\label{eq:Dnstar-substitution1}
y_0=x_1, \quad y_i=\frac{x_{i+1}-x_1}{1-2x_1} \quad (i=1,\dots,n-1).
\end{equation}
It is easily checked that  the transformation
\eqref{eq:Dnstar-substitution1}
has Jacobian determinant $\pm(1-2y_0)^{n-1}$ and maps the
region \eqref{eq:Dnstar-characterization} to the 
set of $n$-tuples $(y_0,\dots,y_{n-1})$ satisfying
\begin{equation}
\label{eq:Dnstar-characterization2}
0\le y_0<1/2,\quad 0\le y_1<1, \quad 0\le y_i<1-y_{i-1}\quad
(i=2,\dots,n-1).
\end{equation}
It follows that 
\begin{align} 
\label{eq:Dnstar-vol-calculation2}
   \vol(\Dnstar) &= \int_0^{1/2} (1-2y_0)^{n-1} \; dy_0 
    \int_{0}^{1} \int_{0}^{1-y_1} \cdots 
    \int_{0}^{1-y_{n-2}} dy_{n-1} \cdots dy_2 dy_1 
    \\
    \notag
    &= \frac 1 {2n} \int_{0}^{1} 
    \int_{0}^{1-y_1} \cdots 
    \int_{0}^{1-y_{n-2}} dy_{n-1} \cdots dy_2 dy_1
    =\frac{1}{2n}\vol(E_{n-1}),
\end{align}
where 
\begin{align*}
\label{eq:En-def}
E_{n-1}&=\Bigl\{(y_1,\dots,y_{n-1})\in [0,1]^{n-1}: 0\le y_1<1,
0\le y_i<1-y_{i-1}\quad (i=2,\dots,n-1)\Bigr\}.
\end{align*}

Now set, for $i=1,\dots,n-1$,  
\begin{equation}
\label{eq:Dnstar-substitution2}
u_i=\begin{cases} y_i&\text{if $i$ is odd,}
\\
1-y_{i} &\text{if $i$ is even.}
\end{cases}
\end{equation}
The transformation \eqref{eq:Dnstar-substitution2}
has Jacobian determinant $\pm 1$ and thus is volume-preserving.
Moreover, noting that the condition
$y_i<1-y_{i-1}$ is equivalent to $u_i< u_{i-1}$ when $i$ is odd, and to
$u_{i-1}<u_i$ when $i$ is even, we see that this transformation 
maps the region $E_{n-1}$ to the region 
\begin{equation*}
\label{eq:Enstar-def}
E_{n-1}^*=\{(u_1,\dots,u_{n-1})\in [0,1]^{n-1}: u_1<u_2>u_3<\dots \}.
\end{equation*}
The set $E_{n-1}^*$ is, up to a set of volume $0$,
the set of tuples in the $(n-1)$-dimensional unit cube
whose coordinates form an up-down alternating permutation of length $n-1$.  
Since there are $(n-1)!$ permutations of the coordinates 
and, by symmetry, each such permutation contributes an amount
$1/(n-1)!$ to the volume of the unit cube,  it follows that 
\begin{equation}
\label{eq:Enstar-vol}
\vol(E_{n-1})=\vol(E_{n-1}^*)=\frac{A_{n-1}}{(n-1)!},
\end{equation}
where $A_{n-1}$ is the number of up-down alternating permutations of 
length $n-1$. 
Combining \eqref{eq:Dnstar-vol-calculation2}
with \eqref {eq:Enstar-vol} we obtain
\begin{equation*}
\vol(\Dnstar)=\frac{1}{2n}\vol(E_{n-1})=\frac1{2n}\vol(E_{n-1}^*)
=\frac{A_{n-1}}{(2n)(n-1)!}.
\end{equation*}
This is the desired formula \eqref{eq:vol-Dn-formula}.
\end{proof}

\begin{lem}
\label{lem:An-bound}
The number $A_n$ of alternating permutations satisfies
\begin{equation}
\label{eq:An-bound}
\frac{A_n}{n!}\le
3\left(\frac{2}{\pi}\right)^{n+1} 
\quad (n\ge 1).
\end{equation}
\end{lem}

\begin{proof}
Andr\'e \cite[p. 23]{andre1881}
showed that, for any integer $m\ge 1$,
\begin{align}
\label{eq:An-bound-even}
\frac{A_{2m}}{(2m)!}
&=2
\left(\frac{2}{\pi}\right)^{2m+1}
\left(1-
\frac1{3^{2m+1}}
+\frac1{5^{2m+1}}
-\frac1{7^{2m+1}}+\dots\right),
\\
\label{eq:An-bound-odd}
\frac{A_{2m-1}}{(2m-1)!}
&=2
\left(\frac{2}{\pi}\right)^{2m}
\left(1+
\frac1{3^{2m}} +\frac1{5^{2m}}
+\frac1{7^{2m}}+\dots\right).
\end{align}
The desired bound \eqref{eq:An-bound} follows as the series in
parentheses in \eqref{eq:An-bound-even} is bounded above by $1$, while
the series in \eqref{eq:An-bound-odd} is at most
\begin{equation*}
\le 1+\sum_{n=3}^\infty
\frac{1}{n^2}=-\frac{1}{4}+\frac{\pi^2}{6}<\frac32.
\qedhere
\end{equation*}
\end{proof}

\begin{proof}[Proof of Theorem \ref{thm:main2}, lower bound] 
Combining Lemmas \ref{lem:vol-Dn}, \ref{lem:vol-Dn-formula},
and \ref{lem:An-bound} 
we obtain
\begin{align*}
p_n=\vol(\Cn) &\ge 1-2n\vol(\Dnstar)
=1-2n\frac{A_{n-1}}{(2n)(n-1)!}
\\
&=1-\frac{A_{n-1}}{(n-1)!}
\ge 1-3\left(\frac{2}{\pi}\right)^{n},
\end{align*}
which is the desired lower bound for $p_n$.
\end{proof}

\section{Proof of Theorem \ref{thm:main3}}
\label{sec:proof-theorem3}

Let
\begin{equation}
        \label{eq:ordered-cyclic-set}
\Cthreeo=
	\left\{ \xyz \in [0,1]^3: \xyz\text{ is cyclic and
	$x \le y \le z$} \right\},
\end{equation}
i.e., $\Cthreeo$ is the set of triples in $\Cthree$ with nondecreasing
coordinates.

The following lemma characterizes the set $\Cthreeo$ 
in terms of inequalities on $x$, $y$, and $z$. Recall (see \eqref{eq:omega-def}
and \eqref{eq:omega-relation}) that $\omega=(\sqrt{5}-1)/2$ is the
positive root of $\omega^2+\omega=1$.

\begin{lem}
    \label{lem:ordered-cyclic-characterizations}
We have:    
\begin{itemize}
        \item[(i)] A triple $\xyz$ belongs to the set $\Cthreeo$
	if and only if it satisfies
        \begin{equation}
        \label{eq:smallest-variable-characterization}
        \left\{
        \begin{aligned}
            0 \le &x \le \omega \\
            x \le &y \le \sqrt {1-x} \\
            \max (y, (1-x)(1-y)) \le &z \le \min \left( 1, \frac {1-x} y \right)
        \end{aligned}
        \right\} .
        \end{equation}
        \item[(ii)] A triple $\xyz$ belongs 
	to the set $\Cthreeo$ if and only if it satisfies
    \begin{equation}
        \label{eq:middle-variable-characterization}
        \left\{
        \begin{aligned}
            0 \le &y \le 1 \\
            0 \le &x \le \min \left( y, 1-y^2 \right) \\
            \max (y, (1-x)(1-y)) \le &z \le \min \left(1, \frac{1-x} y \right)
        \end{aligned}
        \right\} .
    \end{equation}
    \end{itemize}
Moreover, the intervals for the variables $x$, $y$, and $z$ in 
\eqref{eq:smallest-variable-characterization}
and \eqref{eq:middle-variable-characterization} are always non-empty,
i.e., the upper bounds on these variables are always greater or equal to
the corresponding lower bounds.
\end{lem}

\begin{proof}
By Lemma \ref{lem:cyclic-characterizations}(i), a triple
$\xyz\in[0,1]^3$ with $x\le y\le z$ is cyclic if and only if it
satisfies the inequalities 
\begin{equation}
\label{eq:cyclic-bound-simplified3}
x+yz\le 1\text{ and } (1-z)+(1-y)(1-x)\le 1.
\end{equation}
Hence $\Cthreeo$ is the set of triples $\xyz\in[0,1]^3$ satisfying 
$x\le y\le z$ and \eqref{eq:cyclic-bound-simplified3}.

Consider now a triple $\xyz\in[0,1]^3$ satisfying these conditions.
Then $x+x^2 \le x+yz\le 1$, and hence $x\le \omega$, with 
$\omega=(\sqrt{5}-1)/2$ defined as above.  Thus, we must have
\begin{equation}
\label{eq:x-bound1}
0\le x\le \omega.
\end{equation}
Next, note that  $x+y^2\le x+yz\le 1$ and hence $y\le \sqrt{1-x}$. 
Since we also have $x \le y \le z$, it follows that
\begin{equation}
\label{eq:y-bound1}
x\le y\le \sqrt{1-x}.
\end{equation}
Note that since $x\le \omega$, with $\omega$ being the positive root of 
$\omega^2+\omega=1$, we have $x^2+x\le 1$ and hence $x\le \sqrt{1-x}$.
Thus, the interval for $y$ in \eqref{eq:y-bound1} is non-empty.

Next, we have trivially
\begin{equation}
\label{eq:y-bound2}
0\le y\le 1,
\end{equation}
while the inequalities $x\le y\le z$ and $x+y^2\le x+yz\le 1$ 
imply 
\begin{equation}
\label{eq:x-bound2}
0\le x\le \min(y,1-y^2).
\end{equation}
Since $0\le y\le 1$, 
the interval \eqref{eq:x-bound2} is clearly non-empty. 

Finally, the inequalities \eqref{eq:cyclic-bound-simplified3}
imply $z\le (1-x)/y$ and 
$z\ge (1-y)(1-x)$, which combined with $0\le z\le 1$ and $x\le y\le z$ 
yields
\begin{equation}
\label{eq:z-bound}
\max (y, (1-x)(1-y)) \le z \le \min \left(1, \frac{1-x} y \right).
\end{equation}
Again the interval for $z$ here is non-empty since, by
\eqref{eq:y-bound1}, $(1-x)/y\ge y$, while we also have 
$(1-x)/y\ge (1-x)\ge (1-x)(1-y)$, $1\ge y$, and $1\ge (1-x)(1-y)$ which follow from $0 \le x,y \le 1$.

Note that the inequalities
\eqref{eq:x-bound1}, \eqref{eq:y-bound1}, and \eqref{eq:z-bound}
are exactly those in the desired characterization
\eqref{eq:smallest-variable-characterization} of $\Cthreeo$, 
while \eqref{eq:y-bound2},
\eqref{eq:x-bound2}, 
and \eqref{eq:z-bound}
are those in \eqref{eq:middle-variable-characterization}.
Thus, we
have shown that any triple $\xyz\in \Cthreeo$ must satisfy 
\eqref{eq:smallest-variable-characterization}
and 
\eqref{eq:middle-variable-characterization}
and that
the intervals for $x$, $y$, and $z$ in
\eqref{eq:smallest-variable-characterization}
and \eqref{eq:middle-variable-characterization}
are all non-trivial.

Conversely, assume that $\xyz$ satisfies
\eqref{eq:smallest-variable-characterization} or 
\eqref{eq:middle-variable-characterization}. 
Then $0\le x\le y\le z\le 1$, and $z$ must satisfy \eqref{eq:z-bound}.
The upper bound for $z$ in \eqref{eq:z-bound} implies 
$z\le (1-x)/y$ and hence $x+yz\le 1$,
i.e., the first inequality in \eqref{eq:cyclic-bound-simplified3}.
The lower bound for $z$ in \eqref{eq:z-bound} implies
$(1-x)(1-y)\le z$ and hence $(1-z)+(1-x)(1-y)\le 1$, i.e., the second inequality in
\eqref{eq:cyclic-bound-simplified3}. 
Thus, $\xyz$ must be an element of $\Cthreeo$.

This completes the proof of the lemma.
\end{proof}

\begin{proof}[Proof of Theorem \ref{thm:main3}]
Since, by Lemma \ref{lem:symmetry-triples}, 
the ``cyclic triple'' property is invariant with respect to taking
permutations, we have
\begin{equation}
\label{eq:ordered-cyclic-volume}
\vol(\Cthreeo)= \frac{1}{3!}\vol(\Cthree)=\frac16 p_3.
\end{equation}
Moreover, the distribution of the order statistics $(X_1^*,X_2^*,X_3^*)$
of a random triple in $\Cthree$ is the same as the distribution of 
a random vector $(X,Y,Z)$ that is distributed uniformly over the region 
$\Cthreeo$.  It follows that the density functions $f_1$, $f_2$,
and $f_3$ in Theorem \ref{thm:main3} are the marginal densities of the 
latter vector and thus can be obtained by integrating over appropriate
slices of the region $\Cthreeo$ and dividing by the volume of this region.
That is, we have
\begin{equation}
\label{eq:f1-definition2}
f_1(x)= \frac{1}{\vol(\Cthreeo)} \iint_{A_1(x)}1\,dz\,dy
= \frac{6}{p_3} \iint_{A_1(x)}1\,dz\,dy,
\end{equation}
where 
\begin{equation}
\label{eq:Ax-definition}
A_1(x)=\{(y,z)\in[0,1]^2: \xyz\in\Cthreeo\},
\end{equation}
along with analogous formulas for $f_2$ and $f_3$.

Using the first characterization of $\Cthreeo$ in 
    Lemma \ref{lem:ordered-cyclic-characterizations},
we see that
$f_1(x)$ is supported on the interval $[0,\omega]$, and that for each
$x$ with $0\le x\le \omega$ we have 
\begin{align*}
         f_1(x) &= \frac 6 {p_3} \int_x^{\sqrt{1-x}}
	 \int_{\max (y, (1-x)(1-y))}^{\min \left( 1, \frac {1-x} y \right)}
	 \, dz \, dy
         \\
         &=\frac 6 {p_3} \int_x^{\sqrt{1-x}}
         \left(\min \left( 1, \frac {1-x} y \right) 
	 - \max (y, (1-x)(1-y)) \right) dy
         \\[1ex]
         &= \frac 6 {p_3} \begin{cases}
            \frac 1 2 \left( x^3-3x^2 + 
	    \frac {1-x} {2-x} - (1-x) \ln (1-x) \right)
	    &\text{if $0 \le x \le 1-\omega $,}
            \\[0.5ex]
            \frac 1 2 \left( x^2-3x+1 - (1-x)\ln (1-x) \right)
	    &\text{if $1-\omega \le x \le 1/2$,}
	    \\[0.5ex]
            \frac 1 2 \left( x^2+x-1+(1-x)
	    \ln (1-x)-2(1-x)\ln x \right) &\text{if $1/2 \le x \le
	    \omega$.}
        \end{cases}
\end{align*}
This yields the desired formula \eqref{eq:min-density}
on noting that $\omega=(1-\sqrt{5})/2$ and 
\begin{equation}
\label{eq:omega-relation2}
1-\omega=1-\frac{\sqrt{5}-1}{2}=\frac{3-\sqrt{5}}{2}.
\end{equation}

Similarly, using the second characterization of 
    Lemma \ref{lem:ordered-cyclic-characterizations},
we obtain
\begin{align*}
         f_2(y) &= 
	 \frac 6 {p_3} \int_0^{\min \left( y,1-y^2 \right)}
	 \int_{\max (y, (1-x)(1-y))}^{\min \left( 1, \frac {1-x} y \right)}
	 \, dz\, dx
         \\
         &= \frac 6 {p_3} \int_0^{\min \left( y,1-y^2 \right)}
	 \left( \min \left( 1, \frac {1-x} y \right) 
	 - \max (y, (1-x)(1-y)) \right) dx
         \\[1ex]
         &= \frac 6 {p_3} \begin{cases}
        \frac 1 2 \left( 3y^2-y^3 \right)
	&\text{if $ 0 \le y \le 1-\omega$,}
	\\[0.5ex]
        3y-y^2-\frac 1 {2(1-y)} 
	&\text{if $ 1-\omega \le y \le 1/2$,} 
	\\[0.5ex]
        2-y-y^2- \frac 1 {2y} &\text{if $1/2 \le y \le \omega$,}
	\\[0.5ex]
        \frac 1 2 \left( \frac {y^3} 2 - \frac {3y} 2 + 1 \right)
	&\text{if $\omega \le y \le 1$} 
    \end{cases}
\\[1ex]
    &= \begin{cases}
        \frac 3 {p_3} \left( 3x^2-x^3 \right)
	&\text{if $ 0 \le x \le 1 - \omega$,}
	\\[.5ex]
        \frac 6 {p_3} \left( 3x-x^2-\frac 1 {2(1-x)} \right)
	&\text{if $1 - \omega < x \le 1/2$,}
	\\[.5ex]
	f_2(1-x) &\text{if $1/2 < x \le 1$,}
    \end{cases}
\end{align*}
This yields the desired formula \eqref{eq:middle-density}
for $f_2(y)$.

Finally, since, by Lemma \ref{lem:symmetry-ntuples}(iii), a triple
$(x_1,x_2,x_3)$ is cyclic if and only if $(1-x_1,1-x_2,1-x_3)$ is
cyclic, the distribution of $X_3^*$, the largest element in a cyclic
triple, is the same as that of $1-X_1^*$. Thus, we have $f_3(x)=f_1(1-x)$.

This completes the proof of Theorem \ref{thm:main3}.
\end{proof}

\section{Concluding Remarks}
\bigskip

\label{sec:concluding-remarks}

In this section we mention some related questions and conjectures
suggested by our results.

\paragraph{Precise asymptotic behavior of $p_n$.}
Theorem \ref{thm:main2} implies that
\begin{equation*}
e^{-(1+o(1))c_1 n}\le 1-p_n\le e^{-(1+o(1))c_2 n} 
\quad (n\to\infty)
\end{equation*}
holds with $c_1=\ln 4$ and $c_2=\ln (\pi/2)$.   It seems reasonable to
expect that there exists a \emph{single} constant $c$, with $c_2\le c\le
c_1$, such that 
\begin{equation*}
1-p_n=e^{-(1+o(1))cn}\quad (n\to\infty),
\end{equation*}
i.e., that the limit
\begin{equation*}
c=\lim_{n\to\infty}\frac{-\ln(1-p_n)}{n}
\end{equation*}
exits.  If true, the value of $c$ must lie between the constants
$c_2=\ln(\pi/2)=0.451\dots$ and $c_1=\ln 4=1.386\dots$.

\paragraph{Prescribing all pairwise probabilities $P(U_i>U_j)$.}
By definition, a cyclic $n$-tuple is an $n$-tuple of real numbers in
$[0,1]$ that can represent the $n$ probabilities $P(U_{i+1}>U_i)$,
$i=1,\dots,n$, with a suitable choice of independent random variables
$U_i$, $i=1,\dots,n$, satisfying $P(U_i=U_j)=0$ for $i\not=j$.

As a natural extension of this concept, one can ask which arrays 
$x_{ij}$, $i,j=1,\dots,n$, can represent probabilities of the form
\begin{equation}
\label{eq:xij-representation}
x_{ij}=P(U_j>U_i)\quad (i,j=1,\dots,n),
\end{equation}
under the same assumptions on the random variables $U_i$.  This problem
arises in the theory of Social Choice, where it is known under the term
\emph{utility representations}; see  Suck \cite{suck2002} for background
and references.

Obviously, in order for \eqref{eq:xij-representation} to hold, 
it is necessary that $x_{ii}=0$ for all $i$. Moreover,  
the assumption that $P(U_i=U_j)=0$ for $i\not=j$ implies
$x_{ij}=P(U_j>U_i)=1-P(U_i>U_j)=1-x_{ji}$ for $i>j$. Thus the array
$x_{ij}$ is determined by the $\binom{n}{2}$ values $x_{ij}$, $1\le
i<j\le n$. Considering these values as tuples in the
$\binom{n}{2}$-dimensional unit cube, one can then ask for the
probability that a \emph{random} tuple in this cube has a representation 
of the form \eqref{eq:xij-representation}.

When $n=3$, the values $x_{ij}$ are determined by 
$x_{12}=P(U_2>U_1)$, $x_{23}=P(U_3>U_2)$, and
$x_{13}=P(U_3>U_1)=1-P(U_1>U_3)$. Thus we see that a representation
\eqref{eq:xij-representation} is possible if and only if the triple
$(x_{12},x_{23},1-x_{13})$ is cyclic, so the problem is equivalent to
the problem considered in Theorem \ref{thm:main1}. 

In the general case, however, the two problems are different.  Indeed,
the  question of characterizing arrays $x_{ij}$ that have
representations of the form \eqref{eq:xij-representation} seems to be an
extraordinarily hard problem that remains largely unsolved;  we refer to Suck
\cite{suck2002} for further discussion and some partial results. 

\paragraph{Representations with dependent random variables $U_i$.}
In our definition of cyclic tuples we assumed the random
variables $U_i$ to be independent. This is a natural assumption that is made in most of the mathematical literature on nontransitivity phenomena.  In the context of dice rolls the assumption reflects the fact that if $n$ dice are rolled, then the values showing on the faces of these dice are indeed independent. 

However, there also exist  well-known
nontransitivity paradoxes---especially in voting theory and social choice theory---in which the underlying random variables are dependent. Thus, it would be of interest to study the analogs of cyclic and nontransitive tuples if the 
independence assumption on $U_i$ is dropped.  We refer to Marengo et al. \cite{marengo2020} for some recent results in this direction, and 
to Suck \cite{suck2002}
for an in-depth discussion of the differences between dependent and
independent representations of the form \eqref{eq:xij-representation}.

\bibliographystyle{amsplain}



\providecommand{\bysame}{\leavevmode\hbox to3em{\hrulefill}\thinspace}
\providecommand{\MR}{\relax\ifhmode\unskip\space\fi MR }
\providecommand{\MRhref}[2]{%
  \href{http://www.ams.org/mathscinet-getitem?mr=#1}{#2}
}
\providecommand{\href}[2]{#2}

\end{document}